\documentclass[reqno,11pt]{amsart}
\usepackage{amsfonts}
\usepackage{amsmath,latexsym,amssymb,amsfonts,amsbsy, amsthm}

\renewcommand{\theequation}{\thesection.\arabic{equation}}

\newtheorem*{theorem*}{Theorem}
\newtheorem*{definition*}{Definition}

\newtheorem{definition}{Definition}[section]
\newtheorem{theorem}{Theorem}[section]
\newtheorem{lemma}[theorem]{Lemma}
\newtheorem{proposition}[theorem]{Proposition}

\def \r{\rho}
\def \i{\Im}
\def \th{\theta}
\def \b{\beta}
\def \a{\alpha}
\def \p{\partial}
\def \eps{\varepsilon}
\def \grad{\nabla_{\! x}}
\def \LAP{\Delta_{\!x}}
\def \pO{\partial\Omega}

\begin{document}

\title[Prandtl Equation with Robin Boundary Condition]{ On Analytic Solutions of the Prandtl Equations with Robin Boundary Condition in Half Space}

\author[Y-T. Ding]{Yutao Ding}
\address{Mathematical Sciences Center, Tsinghua University}
\address{Jin Chunyuan West Building, Tsinghua University, Beijing, 100084}
\email{ytding@math.tsinghua.edu.cn}
\author[N. Jiang]{Ning Jiang}
\address{Mathematical Sciences Center, Tsinghua University}
\address{Jin Chunyuan West Building, Tsinghua University, Beijing, 100084}
\email{njiang@tsinghua.edu.cn}

\date{}

%\date{}
\maketitle

%\noindent {\sl AMS Subject Classification (2000):} 35Q30, 76D03  \
\begin{abstract}
 The existence and uniqueness of the analytic solutions to the nonlinear Prandtl equations with Robin boundary condition on a half space are proved, based on an application of abstract Cauchy-Kowalewski theorem. These equations arise in the inviscid limit of incompressible Navier-Stokes equations with Navier-slip boundary condition in which the slip length is square root of viscosity, as formally derived in \cite{wangxin}.
\end{abstract}

\renewcommand{\theequation}{\thesection.\arabic{equation}}
\setcounter{equation}{0}
%%%%%%%%%%%%%%%%%%%%%%%%%%%%%%%%%%%%%%%%%%%%%%
%%%%%%%%%%%%%%%%%%%%%%%%%%%%%%%%%%%%%%%%%

\section{ introduction}
The nonlinear Prandtl equations in the half space $\{(x,Y): x\in \mathbb{R}, Y\in \mathbb{R^+}\}$ are the following system of the tangential component of the velocity $u^P$ and the normal component of velocity $v^P$:
\begin{equation}\label{Prandtl}
\begin{split}
\partial_tu^P+u^P\partial_xu^P+v^P\partial_Yu^P-\partial_{YY} u^P+\partial_x p=&0\,,\\
\partial_xu^P+\partial_Yv^P=&0\,,\\
\partial_Y p =& 0\,.
\end{split}
\end{equation}
In this paper, we impose the Prandtl equations \eqref{Prandtl} with Robin boundary condition:
\begin{equation}\label{Robin BC}
\begin{split}
(u^P-\p_Y u^P)|_{Y=0}=&0,\\
 v^P|_{Y=0}=&0,\\
\lim_{Y\rightarrow\infty}u^P(x,Y,t)=& U(x,t)\,,
\end{split}
\end{equation}
and initial data:
\begin{equation}\label{Prandtl IC}
  u^P(x,Y,t)|_{t=0} = u_0^P(x,Y)\,.
\end{equation}
The third equation in \eqref{Prandtl} implies that the pressure $p$ does not depend on $Y$. Furthermore, $U$ satisfies the Bernoulli equation:
\begin{equation}\label{bernoulli}
\partial_t U + U\partial_x U+\partial_x p=0\,,
\end{equation}
which is the matching condition to the inner layer governed by incompressible Euler equation, i.e $U(x,0)= u^E(x,y,0)|_{y=0}$. Throughout the paper, we use the notation $Y=\frac{y}{\sqrt{\nu}}$, where $\nu$ is the viscosity of the Navier-Stokes equations.

This paper addresses the well-posedness of initial boundary problems \eqref{Prandtl}-\eqref{bernoulli} for analytic data. The main theorem is:\\

\noindent {\bf Main Theorem:} If $U(x,t)$ and $u_0^P(x,Y)$ are analytic, then the Prandtl equations with Robin boundary condition \eqref{Prandtl}-\eqref{bernoulli} have a unique analytic solution. \\

The more detail statement of the main theorem, in particular the functional spaces of $U(x,t)$, $u_0^P(x,Y)$ and the solutions will be given later. First we recall some history of the mathematical theory of Prandtl equations.

Nonlinear Prandtl equations arise in the zero viscosity limit for the incompressible Navier-Stokes equations in a domain with boundary, due to the disparity of the boundary conditions between Navier-Stokes and Euler equations. A thin region, the so-called boundary layer comes out near the boundary in which the values of the unknown functions change drastically in this zero viscosity limit. For the case of non-slip boundary condition, by rescaling the variable normal to the boundary with the scale of the square root of the viscosity, Ludwig Prandtl \cite{p} proposed the so-called Prandtl equations in 1904. Although the Prandtl equations have been simplified to a great extent, as compared with the original Navier-Stokes equations, they are still difficult from the mathematical point of view.

Under the monotone condition $\partial_y u^P>0,$  the Crocco transformation, introduced by Crocco \cite{c}, reduces the boundary layer equations to a single degenerate quasilinear parabolic equation. Based on this transformation, Oleinik and Samokhin \cite{os} proved the local in time well-posedness. Xin and Zhang \cite{xz} proved the existence of global weak solutions.
See also \cite{al,gt,wong} for other results under this monotonic assumption.

For the case without the monotone condition, the well-posedness in the usual Sobelev space, even local in time remains a widely open problem so far. If the data is analytic, Sammartino and Caflisch \cite{sc} proved the short time existence and uniqueness. Sammartino improved in \cite{s} where analyticity is required only in the $x\mbox{-}$axis direction, using the regularizing effect of the viscosity. This result was reproved and only assumed algebraic decay rate of $u^P -U$ by Kukavica and Vicol in \cite{KV} using energy-based method. More recently, G$\acute{e}$rard-Varet and Masmoudi proved the local well-posedness for the data that belongs to the Gevrey class $7/4$ in the horizontal variable $x$. On the other hand, there are also some ill-posedness results for the general Prandtl equations, for example, \cite{ca,e,gd,gd2}. Furthermore, Sammartino and Caflisch \cite{sc1} justified the convergence from the solutions of  Navier-Stokes equations to the Euler equation with Prandtl equations as the error term,  when the data are analytic. More recently, Maekawa \cite{mae} proved this convergence when the vorticity vanished near the boundary.

All above works on the Prandtl equations are about Dirichlet boundary condition, i.e. the first two conditions in \eqref{Robin BC} are replaces by $u^P|_{Y=0}= v^P|_{Y=0}=0$. In view of experimental facts, besides the non-slip boundary condition, some slip conditions can be also imposed on the Navier-Stokes equations. In \cite{navier}, Navier proposed the slip boundary condition which is now called Navier boundary condition. The boundary layer phenomenon is much weaker for this case, comparing to the slip boundary case, which makes the inviscid limit problem is more accessible mathematically. Recently, there have been many studies about the vanishing viscosity limit problem for Navier boundary condition, see \cite{1,2,clopau,lopes,plana,plana1,rousset,xin,xin1}. Note that in these works, the slip length is independent of viscosity.

In \cite{wangxin}, Wang-Wang-Xin studied the asymptotic behavior of the solutions
to the Navier-Stokes equations with Navier boundary condition in which the slip length depending on the viscosity. More specifically, they consider the incompressible Navier-Stokes equation of $\mathrm{u}^\nu = (u^\nu, v^\nu)$  in $\Omega \subset \mathbb{R}^2$, where $\Omega$ is a bounded domain or half space.
\begin{equation}\label{NS}
\begin{aligned}
   \p_t \mathrm{u}^\nu + \mathrm{u}^\nu\!\cdot\! \grad \mathrm{u}^\nu + \grad p^\nu = \nu \LAP \mathrm{u}^\nu\,, \quad &\mbox{in}\quad\! \Omega\times \mathbb{R}_+\,,\\
   \mathrm{div}\mathrm{u}^\nu =0\,, \quad &\mbox{in}\quad\! \Omega\times \mathbb{R}_+\,,
\end{aligned}
\end{equation}
with the boundary conditions
\begin{equation}\label{Navier-slip}
\mathrm{u}^\nu\!\cdot\! \mathrm{n}=0\,,\quad \left[\sigma(\mathrm{u}^\nu)\mathrm{n} + \nu^{-\gamma} \mathrm{u}^\nu\right]^{\mathrm{tan}} =0\,,\quad\mbox{on}\quad\! \pO\,,
\end{equation}
where $\nu$ is the viscosity and $\sigma(\mathrm{u}^\nu)=\frac{1}{2}(\grad \mathrm{u}^\nu + (\grad \mathrm{u}^\nu)^T)$ is the rate of the strain tensor, with $\mathrm{n}$ the normal vector on the boundary $\partial\Omega$.

In \cite{wangxin}, by using multi-scale analysis, Wang-Wang-Xin discovered that the behavior of the vanishing viscosity limit for the problem \eqref{NS} with the boundary condition \eqref{Navier-slip} is influenced by the amplitude of the slip length, in other words, by the value of $\gamma$. They formally deduce that $\gamma=\frac{1}{2}$ is critical in determining the boundary layer behavior. When $\gamma$ is super-critical, i.e. $\gamma > \frac{1}{2}$, the leading boundary layer profile satisfies the same boundary problem for the nonlinear Prandtl equations as in the non-slip case. In the critical case $\gamma =\frac{1}{2}$, the boundary layer profile also satisfies the nonlinear Prandtl equations but with a Robin boundary condition for the tangential velocity profile. When $\gamma$ is sub-critical, i.e. $\gamma < \frac{1}{2}$, the boundary layer profile appears in the order $O(\eps^{1-2\gamma})$ terms of solutions, and satisfies a boundary value problem for linearized Prandtl equations.

It remains a big challenging problem to rigorously justify the formal analysis in \cite{wangxin} in Sobolev spaces, in particular for the critical and super-critical cases, because for those cases the leading boundary layer equations are nonlinear Prandtl equations for which the well-posedness is far from well-understood. Even for the case $\gamma < \frac{1}{2}$, although the leading boundary layer obeys a linearized Prandtl equation, the vorticity is unbounded, which makes the justification of the expansion of $\mathrm{u}^\nu$ with respect to the viscosity $\nu$ nontrivial.

Our aim in this program is to consider a simpler problem: to justify the the inviscid limit proposed in \cite{wangxin} for analytic data. In other words, we seek for the analogue of Sammartino-Caflisch's result for the Navier-slip boundary condition. The first step is to prove the well-posedness of the corresponding Prandtl equation. In this paper, we study the boundary layer equation derived in \cite{wangxin} for the critical case $\gamma= \frac{1}{2}$, i.e. the Prandtl equation with Rodin boundary condition. Furthermore, we focus on the simplest case: $\mathrm{u}^\nu(x,y)= (u^\nu(x,y), v^\nu(x,y))$, $\Omega= \mathbb{R} \times \mathbb{R}^+$, and the initial data is analytic. More specifically, we consider the 2-dimensional Navier-Stokes equations with Navier boundary condition on a half-space:
\begin{equation}\nonumber
\begin{split}
 \partial_tu^{\nu}+u^\nu \partial_xu^\nu+v^\nu\partial_yu^\nu+\partial_xp^\nu&= \nu\triangle u^\nu\,, \\
 \partial_tv^{\nu}+u^\nu \partial_xv^\nu+v^\nu\partial_yv^\nu+\partial_yp^\nu &=\nu\triangle v^\nu\,, \\
 \partial_xu^\nu+\partial_yv^\nu&=0,\\
 ( u^\nu-\sqrt{\nu}\p_y u^\nu)|_{y=0}&=0,\\
  v^\nu|_{y=0}&=0,
\end{split}
\end{equation}
 where $(x,y)\in\mathbb{R}\times\mathbb{R}^+,$ $u^\nu$ is the tangential components of velocity to the
 boundary $\Gamma=\{(x,0): x\in \mathbb{R}\}$, and $v^\nu$ is the normal components.
 We may formally write $(u^\nu\,,v^\nu\,,p^\nu)(t,x,y)$  as:
\begin{eqnarray}\nonumber
u^\nu(t,x,y)&=&u^E(t,x,y)+u^B(t,x,\tfrac{y}{\sqrt{\nu}}),\\
v^\nu(t,x,y)&=&v^E(t,x,y)+\sqrt{\nu}v^B(t,x,\tfrac{y}{\sqrt{\nu}}),
\end{eqnarray}
and
\begin{equation}\nonumber
p^\nu(t,x,y)=p^E(t,x,y)+p^B(t,x,\tfrac{y}{\sqrt{\nu}}).
\end{equation}
Denote by $Y=\frac{y}{\sqrt{\nu}},$ and we define:
\begin{eqnarray}\nonumber
u^P(t,x,Y)&:=&u^E(t,x,0)+u^B(t,x,Y),\\
v^P(t,x,Y)&:=&Y\partial_y v^E(t,x,0)+v^B(t,x,Y)\,.
\end{eqnarray}

In \cite{wangxin}, it was formally derived that as $\nu$ tends to zero, $(u^E, v^E)$ satisfies the incompressible Euler equations, and  $(u^P, v^P)$ satisfies approximately the nonlinear Prandtl equation with Robin boundary condition, \eqref{Prandtl}-\eqref{bernoulli}, with $U(x,t) = u^E(x,0,t)$, as stated in the beginning of the paper.

In the present paper, we show the existence and the uniqueness of \eqref{Prandtl}-\eqref{bernoulli} in the analytic data. The main strategy is basically the same as \cite{sc}. We first introduce a new variable for which the Prandtl equations can be rewritten as a nonlinear heat equation with Robin boundary condition. Then we invert the heat operator, taking into account boundary and initial conditions. After estimating some bounds for this heat operators, we can verify that the Abstract Cauchy-Kowalewski theorem is applicable. Comparing to \cite{sc}, the main novelty of this paper is that since the boundary condition at $Y=0$ is mixed type for which the formula for the heat kernel is not easy to derive. We introduce a further new variable to transfer the Robin condition to Neumann boundary condition. As a price, a new first order $Y\mbox{-}$derivative term appears which need some new estimates. We emphasize that in our next paper in this program, i.e. justifying the inviscid limit of the Navier-Stokes equations with the Navier-slip boundary condition in which the slip length is the square root of the viscosity, this first order $Y\mbox{-}$derivative term will introduce more serious difficulty.

This paper is organized as follows: the next section will devoted to the introduction of functional spaces in which the theorem will be proved, the statement of the version of abstract Cauchy-Kowalewski theorem we will use, and the main theorem of the existence of the Prandtl equations. In Section 3 we estimate the heat operator. Applying the abstract Cauchy-Kowalewski theorem to prove the main theorem is in Section 4.

\section{The Main Result}

 \subsection{Function Spaces}
In this section, we introduce the function spaces used in the proof of the existence and uniqueness of the Prandtl equations. These function space
were defined and used in \cite{sc} and \cite{sc1}. We first define the following strip and angular sector in the complex plane.
$$D(\r)=\{ z\in \mathbb{C}: \i z \in(-\r,\r)\},$$
$$\Sigma(\th)=\{z \in\mathbb{C}:\Re z\geq0,\;|\i z|\leq \Re z \tan \th\}\,,$$
where $\Re z$ and $\Im z$ denote the real and imaginary parts of the complex number $z$ respectively. We also define for $b\in \mathbb{R}$
$$\Gamma(b)=\{z \in \mathbb{C}:\i z=b\}.$$
For a family of Banach spaces $\{X_{\r}\}_{0\leq\r\le \r_0},$ we denote $B^k_{\beta}([0,T],X_{\rho})$
as the space of all $C^k$ functions from $[0,T]$ to $X_{\r}$ with the norm
$$|f|_{k,\r,\b}=\sum_{j=0}^k\sup_{0\leq t\leq T}|\p^j_tf(t)|_{\r-\b t}.$$

Next, we introduce some functional spaces. The first is a space of the functions of $x\mbox{-}$variable.
\begin{definition}
 The space $H^{l,\r}$ is defined as the set of all complex functions $f(x)$ such that
$f$ is analytic in $D(\r)$, $\p^{\a}_x f\in L^2(\Gamma(\i x))$ for $\i x\in (-\r,\r),\;\a \leq l,$
and with the norm
$$|f|_{l,\r}=\sum_{\a \leq l}\sup_{\i x\in(-\r,\r)}\|\p^{\a}_x f(\cdot+i\i x)\|_{L^2(\Gamma(\i x))}<\infty.$$
\end{definition}

The next is a space of functions depending on $x$ and $Y$,
\begin{definition}
 The space $K^{l,\r,\th,\mu}$ with $\mu>0$ is the set of all complex functions $f(x, Y)$ such that
$f$ is analytic inside $D(\r)\times\sum(\th)$, $\p^{\a_1}_Y\p^{\a_2}_x f(x,Y)\in C^0(\sum(\th); H^{0,\r})$ for $\a_1\leq 2,\a_1+\a_2\leq l$
and with the norm
$$|f|_{l,\r,\th,\mu}=\sum_{\a_1 \leq 2}\sum_{\a_2\leq l-\a_2}\sup_{Y\in \Sigma(\th)}e^{\mu \Re Y}|\p^{\a_1}_Y\p^{\a_2}_x f(\cdot\,,Y)|_{0,\r}<\infty.$$
\end{definition}

We also need the next two function spaces depending on $t$:

\begin{definition}
 The space $K^{l,\r}_{\b,T}$  is defined as
 $$K^{l,\r}_{\b,T}=\bigcap^1_{j=0}B^j_{\b}([0,T],H^{l-j,\r})$$
with the norm
$$|f|_{l,\r,\b,T}=\sum_{j=0}^1\sum_{\a\leq l-j}\sup_{0\leq t\leq T}|\p^{j}_t\p^{\a}_x f(t\,,\cdot)|_{0,\r-\b t} < \infty\,.$$
\end{definition}

\begin{definition}
 The space $K^{l,\r,\th,\mu}_{\b,T}$  is defined as the set of functions $f(x,Y,t)$
such that
$$f\in C^0([0,T],K^{l,\r,\th,\mu}),\;\p_t\p^{\a}_x f\in C^0([0,T],K^{0,\r,\th,\mu}),\,\text{with}\,\a\leq l-2,$$
with the norm
\begin{equation*}
\begin{split}
|f|_{l,\r,\th,\mu,\b,T}=&\sum_{\a_1\leq 2}\sum_{\a_1+\a_2\leq l}\sup_{0\leq t\leq T}
|\p^{\a_1}_Y\p^{\a_2}_x f(\cdot\,,Y,t)|_{0,\r-\b t,\th-\b t,\mu-\b t}\\
&+\sum_{\a\leq l-2}\sup_{0\leq t\leq T}|\p_t\p^{\a}_x f(\cdot\,,\cdot\,,t)|_{0,\r-\b t,\th-\b t,\mu-\b t}<\infty.
\end{split}
\end{equation*}
\end{definition}

The following Sobolev inequality is useful and its proof could be found in \cite{sc}.
\begin{proposition}\label{proposition1}
\cite{sc} Let $f\in K^{l,\r,\th,0}_{\b,T}, g\in K^{l,\r,\th,\mu}_{\b,T}$ and $l\geq 3.$
Then  $f\cdot g\in K^{l,\r,\th,\mu}_{\b,T},$ and
\begin{equation}
|f\cdot g|_{l,\r,\th,\mu,\b,T}\leq C|f|_{l,\r,\th,0,\b,T}|g|_{l,\r,\th,\mu,\b,T}.
\end{equation}
\end{proposition}

The following two lemmas is the Cauchy estimate of $H^{l,\r}$ and $K^{l,\r,\th,\mu}$
which have been used in \cite{sc}.
\begin{lemma}\label{lemma1}
\cite{sc}
Let $f\in H^{l,\r''}.$ If $\r'<\r''$ then
\begin{equation}
|\p_x f|_{l,\r''}\leq C\frac{|f|_{l,\r}}{\r''-\r'}.
\end{equation}
\end{lemma}
\begin{lemma}\label{lemma2}
\cite{sc}
Let $f\in K^{\l,\r',\th'',\mu''}.$ Then
\begin{equation}
|\chi(Y)\p_Yf|_{l,\r',\th',\mu'}\leq C(\frac{|f|_{l,\r',\th'',\mu'}}{\th''-\th'}
+\mu'|f|_{l,\r',\th',\mu'}),
\end{equation}
\begin{equation}
|Y\p_Yf|_{l,\r',\th',\mu'}\leq C(\frac{|f|_{l,\r',\th'',\mu'}}{\th''-\th'}
+\mu'\frac{|f|_{l,\r',\th',\mu''}}{\mu''-\mu'}+|f|_{l,\r',\th',\mu'}).
\end{equation}
where $\chi(Y)$ is analytic in $\Sigma(\th'), \chi(0)=0,$ and $\p_Y^j\chi(Y)$ is
bounded in  $\Sigma(\th')$ for $j\leq 2.$
\end{lemma}

\subsection{The Abstract Cauchy-Kowalewski Theorem}
Next, we state the abstract Cauchy-Kowalewski theorem (briefly denoted as  ACK ) of the following version, for the proof see \cite{s}.
\begin{theorem}
Let $\{X_{\r}: 0<\r\leq\r_0\}$ be a family of Banach spaces with norm $|\cdot|_\rho$, such that
$X_{\r'}\subset X_{\r''}$ and $|\cdot|_{\r''}\leq |\cdot|_{\r'}$
for $0<\r''\leq \r'\leq \r_0.$ Consider the equation
\begin{equation}\label{ack}
u+F(t,u)=0,\;\;\;t\in[0,T].
\end{equation}
Suppose that there exists $R>0$, $\r_0>0$ and $\beta_0>0$ such that if $0 < t \leq \frac{\rho_0}{\beta_0}$, the following properties
hold:
\begin{itemize}
\item For any $0<\r'<\r\leq\r_0-\b_0T,$ and $u$ satisfying $\{u(t)\in X_{\r}, \text{and}\;|u(t)|_{\r}\leq R,\; 0\leq
t\leq T\}$, the map $F(t,u):[0,T]\mapsto X_{\r'}$ is continuous.

\item For any  $0<\r<\r_0-\b_0T$ the function $F(t,0):[0,\frac{\r_0}{\b_0}]\mapsto\{u\in X_{\r}:\sup_{0\leq t\leq T}|u(t)|_{\r}\leq R\}$
is continuous and
\begin{equation}\label{ack1}
|F(t,0)|_{\r_0-\b_0t}\leq R_0<R.
\end{equation}

\item For any  $0<\r'<\r(s)<\r_0-\b_0s$ and for any  $ u^1,u^2\in\{u\in X_{\r}:\sup_{0\leq t\leq T}|u(t)|_{\r-\b_0t}\leq R\},$
\begin{equation}\label{ack2}
|F(t,u^1)-F(t,u^2)|_{\r'}\leq C\int^t_0 ds(\frac{|u^1-u^2|_{\r(s)}}{\r(s)-\r'}+\frac{|u^1-u^2|_{\r'}}{\sqrt{t-s}})
\end{equation}
\end{itemize}

Then there exists $\b>\b_0,\,\bar{\r}<\r_0,$ such that for any $ 0<\r<\bar{\r}-\b t,$
the equation \eqref{ack} has a unique solution $u(t)\in X_{\r}$, with $t\in [0,\frac{\bar{\r}}{\b}]$, satisfying
\begin{equation}\label{ack3}
\sup_{t\in [0,\frac{\bar{\r}}{\b}]}|u(t)|_{\bar{\r}-\b t}\leq C.
\end{equation}
\end{theorem}

This theorem is also applicable if the parameter $\r$ is replaced by a vector
of parameters $(\r,\th,\mu)$, and the right side of \eqref{ack2}
should be changed into
\begin{equation}
\begin{split}
&\int^t_0 ds(\frac{|u^1-u^2|_{\r(s),\th',\mu'}}{\r(s)-\r'}+\frac{|u^1-u^2|_{\r',\th(s),\mu'}}{\th(s)-\th'}\\
&\;\;\;\;\;\;\;\;\;\;+\frac{|u^1-u^2|_{\r',\th',\mu(s)}}{\mu(s)-\mu'}+\frac{|u^1-u^2|_{\r',\th',\mu'}}{\sqrt{t-s}}),
\end{split}
\end{equation}
and this case will be used in the proof of our main theorem.

\subsection{The Statement of the Main Result}

Now we are ready to state the main theorem of this paper.
\begin{theorem}\label{th1}
Let $U(x,t)\in K^{l,\r}_{\b,T}$ for $l\geq 3$, and $u_0^P$ with $u_0^P-U(x,0)\in K^{l,\r,\th,\mu}$ satisfies the compatibility conditions
\begin{equation}\nonumber
(\p_Yu_0^P-u_0^P)\mid_{Y=0}=0\,,\quad\mbox{and}\quad\! \lim_{Y\rightarrow \infty}u_0^P(x,Y,t) = U(x,t)\,.
\end{equation}
Then there exists a unique solution $u^P$ of the Prandtl equations \eqref{Prandtl}-\eqref{bernoulli}.
This solution can be written as
\begin{equation}
u^P(x,Y,t)=\tilde{u}^P+U(x,t),
\end{equation}
where $\tilde{u}^P\in K^{l,\r_1,\th_1,\mu_1}_{\b_1,T}$ for $0<\r_1<\r$, $0<\th_1<\th$, $0<\mu_1<\mu$ and $\b_1>\b.$ Furthermore, this solution satisfies the estimate in $K^{l,\r_1,\th_1,\mu_1}_{\b_1,T}$:
\begin{equation}
|\tilde{u}^P|_{l,\r_1,\th_1,\mu_1,\b_1,T}\leq C(|u_0^P-U(x,0)|_{l,\r,\th,\mu}+|U(x,t)|_{l,\r,\b,T}).
\end{equation}
\end{theorem}

Following the same procedure used in \cite{sc}, we shall recast the Prandtl equations in a form suitable for the application of the ACK theorem. First, to get rid of the pressure, we introduce the new variable $\tilde{u}^P$:
\begin{equation}\label{ue}
\begin{split}
\tilde{u}^P=u^P-U(x,t),
\end{split}
\end{equation}
we have that
\begin{equation}
v^P=-\int^Y_0\p_xu^PdY'=-\int^Y_0\p_x\tilde{u}^PdY'-Y\p_xU.
\end{equation}

Then from \eqref{Prandtl}, \eqref{Robin BC} and \eqref{bernoulli}, $\tilde{u}^P$ satisfies:
\begin{equation}\label{p1e}
\begin{split}
(\p_t-\p_{YY})\tilde{u}^P=&-(\tilde{u}^P\p_x\tilde{u}^P+\tilde{u}^P\p_xU+U\p_x\tilde{u}^P)
\\&+(\int^Y_0\p_x\tilde{u}^PdY'+Y\p_xU)\p_Y\tilde{u}^P,\\
\lim_{Y\rightarrow+\infty}\tilde{u}^P=&0,\\
\tilde{u}^P-\p_Y\tilde{u}^P=&-U(x,t)\,,\quad \text{on}\quad\! Y=0\,,\\
\tilde{u}^P\mid_{t=0}=&\tilde{u}^P_0,
\end{split}
\end{equation}
where $\tilde{u}^P_0=u^P_0-U(x,0).$

The main difference of this paper with \cite{sc} is that the boundary condition in \eqref{p1e} is mixed type, which causes new difficulty, in particular that heat kernel with Robin boundary condition is not easy to derive explicitly. To overcome this difficulty,  we introduce the following variable to change the Robin boundary condition into Neumann boundary condition. Let
\begin{equation}
u=e^{-Y}\tilde{u}^P,
\end{equation}
then $u$ solves the following equations:
\begin{equation}\label{p2e}
\begin{split}
(\p_t-\p_{YY})u=&-(e^yu\p_xu+u\p_xU+U\p_xu)+u
\\&+(\int^Y_0e^Y\p_xudY'+Y\p_xU)(\p_Yu+u)+2\p_Yu\\
\doteq & K(u,t)+2\p_Yu,\\
\lim_{Y\rightarrow+\infty}u=&0,\\
\p_Yu=&U(x,t)\,,\quad \text{on}\quad\! Y=0\,,\\
u\mid_{t=0}=&u_0,
\end{split}
\end{equation}
where $u_0=e^{-Y}\tilde{u}^P_0.$  The equation \eqref{p2e} is a heat equation with nonlinear source term, and the boundary condition is inhomogeneous Neumann at $Y=0$., which allows us to use the explicit formula of the heat kernel. However, \eqref{p2e} includes a first order $Y\mbox{-}$ derivative term $2\partial_Y u$, which gives a new difficulty in the estimates. Handling this term is the main difference of this paper with \cite{sc}.

Theorem \ref{th1} is a corollary of the following proposition.
\begin{proposition} \label{pro21}
Assume that $u_0\in K^{l,\r,\th,\mu+1}, U(x,t)\in K^{l,\r}_{\b,T}, l\geq 3,
\p_Yu_0\mid_{Y=0}=U,$
and let $0<\r_1<\r, 0<\th_1<\th, 0<\mu_1<\mu, \b_1>\b.$ Then there exists a unique
solution $u\in K^{l,\r_1,\th_1+1,\mu_1}_{\b_1,T}$ of equation system (\ref{p2e}), and $u$
satisfies:
\begin{equation}
|u|_{l,\r_1,\th_1,\mu_1+1,\b_1,T}\leq C(|u_0|_{l,\r,\th,\mu+1}+|U(x,t)|_{l,\r,\b,T}).
\end{equation}
\end{proposition}

\section{The Estimates of Heat Operator}
This section is devoted to some estimates of heat equation
in the spaces introduced in the last section.
\subsection{The Heat Kernel}
To solve Prandtl equations, we introduce the heat kernel
\begin{equation}
E(Y,t)=\frac{1}{\sqrt{4\pi t}}e^{-\frac{Y^2}{4t}},
\end{equation}
and
\begin{equation}
H(Y,t)=\frac{Y}{t}\frac{1}{\sqrt{4\pi t}}e^{-\frac{Y^2}{4t}}.
\end{equation}

We define the operator $E_1$ as the convolution of $E(Y,t)$
with the even extension to $Y<0$ with the function $u_0(x,Y),$ that is
\begin{equation}
E_1(t)u_0=\int_0^{\infty}[E(Y-Y',t)+E(Y+Y',t)]u_0(x,Y')\,\mathrm{d}Y'\,,
\end{equation}
$E_1(t)u_0$ solves the equation
\begin{equation}
\begin{split}
(\p_t-\p_{YY})E_1(t)u_0&=0,\\
\p_YE_1(t)u_0\mid_{Y=0}&=0,\\
E_1(t)u_0\mid_{t=0}&=u_0(x,Y).
\end{split}
\end{equation}

The operator $E_2$ is defined as
\begin{equation}
\begin{split}
E_2\phi=-\int^{\infty}_Y \int^t_0 H(Y',t-s)\phi(x,s)\,\mathrm{d}s\,\mathrm{d}Y'\,.
\end{split}
\end{equation}
It solves the equation
\begin{equation}
\begin{split}
(\p_t-\p_{YY})E_2\phi&=0,\\
\p_YE_2\phi\mid_{Y=0}&=\phi(x,t),\\
E_2\phi\mid_{t=0}&=0.
\end{split}
\end{equation}

The operator $E_3$ is defined as:
\begin{equation}
E_3\varphi=\int^t_0 \int^{\infty}_0[E(Y-Y',t-s)+E(Y+Y',t-s)]\varphi(Y',s)\,\mathrm{d}Y'\,\mathrm{d}s\,.
\end{equation}
It solves the equation
\begin{equation}
\begin{split}
(\p_t-\p_{YY})E_3\varphi&=\varphi(x,Y,t),\\
\p_YE_3\varphi\mid_{Y=0}&=0,\\
E_3\varphi\mid_{t=0}&=0.
\end{split}
\end{equation}

\subsection{Estimates of $E_1$}
\begin{lemma}\label{dYu}
Let $u\in K^{l,\r,\th,\mu}$ with $\p_Yu\mid_{Y=0}=0,$ then we have
that $E_1(t)u\in K^{l,\r,\th,\mu}$ for all $t>0$, and
\begin{equation}
\sup_{0\leq t\leq T}|E_1(t)u|_{l,\r,\th,\mu}\leq C|u|_{l,\r,\th,\mu}.
\end{equation}
\end{lemma}
\begin{proof}
For $\a\leq l,$ set $\eta=\frac{Y'-Y}{\sqrt{4t}},\zeta=\frac{Y'+Y}{\sqrt{4t}},$ we have that
\begin{equation}
\begin{split}
&\sup_{Y\in\Sigma(\th)}e^{\mu\Re Y}|\p_x^{\a}E_1(t)u|_{0,\r}\\=&\sup_{Y\in\Sigma(\th)}e^{\mu\Re Y}|E_1(t)\p_x^{\a}u|_{0,\r}\\
=&\sup_{Y\in\Sigma(\th)}e^{\mu\Re Y}|\int^{\infty}_0(\frac{1}{\sqrt{4\pi t}}e^{-\frac{(Y-Y')^2}{4t}}+
\frac{1}{\sqrt{4\pi t}}e^{-\frac{(Y+Y')^2}{4t}})\p_x^{\a}u(x,Y')dY'|_{0,\r}\\
=&\sup_{Y\in\Sigma(\th)}e^{\mu\Re Y}|\int^{\infty}_{-\frac{Y}{\sqrt{4t}}}e^{-\eta^2}\p_x^{\a}u(x,Y+\sqrt{4t}\eta)d\eta
+\int^{\infty}_{\frac{Y}{\sqrt{4t}}}e^{-\zeta^2}\p_x^{\a}u(x,-Y+\sqrt{4t}\zeta)d\zeta|_{0,\r}\\
\leq &\sup_{Y\in\Sigma(\th)}|\int^{\infty}_{-\frac{Y}{\sqrt{4t}}}e^{\mu(Y+\sqrt{4t}\eta)}e^{-\eta^2-\mu\sqrt{4t}\eta}
\p_x^{\a}u(x,Y+\sqrt{4t}\eta)d\eta|_{0,\r}\\
&+\sup_{Y\in\Sigma(\th)}|\int^{\infty}_{\frac{Y}{\sqrt{4t}}}e^{\mu(-Y+\sqrt{4t}\zeta)}e^{-\eta^2-\mu\sqrt{4t}\eta+2\mu Y}
\p_x^{\a}u(x,Y+\sqrt{4t}\zeta)d\zeta|_{0,\r}\\
\leq&\sup_{Y\in\Sigma(\th)}e^{\mu\Re Y}|\p_x^{\a}u(x,Y)|_{0,\r}
(\int^{\infty}_{-\frac{Y}{\sqrt{4t}}}e^{-\eta^2-\mu\sqrt{4t}\eta}d\eta+\int^{\infty}_{\frac{Y}{\sqrt{4t}}}e^{-\eta^2-\mu\sqrt{4t}\eta+2\mu Y}
d\zeta)\\
\leq &C \sup_{Y\in\Sigma(\th)}e^{\mu\Re Y}|\p_x^{\a}u(x,Y)|_{0,\r},
\end{split}
\end{equation}
Here we use the fact that for $\th\leq \pi/6,$ there exists a constant $c$, such that
$|e^{-\eta^2}|\leq E^{-c \Re \eta^2}$, then we deduce that
\begin{equation}
\mid\int^{\infty}_{-\frac{Y}{\sqrt{4t}}}e^{-\eta^2-\mu\sqrt{4t}\eta}d\eta+\int^{\infty}_{\frac{Y}{\sqrt{4t}}}e^{-\eta^2-\mu\sqrt{4t}\eta+2\mu Y}
d\zeta\mid\leq C,
\end{equation}
where $C$ is independent on $Y$. For $\a\leq l-1,$
\begin{equation}
\begin{split}
&\sup_{Y\in\Sigma(\th)}e^{\mu\Re Y}|\p_Y\p_x^{\a}E_1(t)u|_{0,\r}\\
=&\sup_{Y\in\Sigma(\th)}e^{\mu\Re Y}|\int^{\infty}_0\p_Y(\frac{1}{\sqrt{4\pi t}}e^{-\frac{(Y-Y')^2}{4t}}+
\frac{1}{\sqrt{4\pi t}}e^{-\frac{(Y+Y')^2}{4t}})\p_x^{\a}u(x,Y')dY'|_{0,\r}\\
=&\sup_{Y\in\Sigma(\th)}e^{\mu\Re Y}|\int^{\infty}_0\p_{Y'}(-\frac{1}{\sqrt{4\pi t}}e^{-\frac{(Y-Y')^2}{4t}}+
\frac{1}{\sqrt{4\pi t}}e^{-\frac{(Y+Y')^2}{4t}})\p_x^{\a}u(x,Y')dY'|_{0,\r}\\
=&\sup_{Y\in\Sigma(\th)}e^{\mu\Re Y}|-\int^{\infty}_0(-\frac{1}{\sqrt{4\pi t}}e^{-\frac{(Y-Y')^2}{4t}}+
\frac{1}{\sqrt{4\pi t}}e^{-\frac{(Y+Y')^2}{4t}})\p_{Y'}\p_x^{\a}u(x,Y')dY'|_{0,\r}\\
\leq &C \sup_{Y\in\Sigma(\th)}e^{\mu\Re Y}|\p_Y\p_x^{\a}u|_{0,\r},
\end{split}
\end{equation}
For $\a\leq l-2,$ we have that
\begin{equation}
\begin{split}
&\sup_{Y\in\Sigma(\th)}e^{\mu\Re Y}|\p_{YY}\p_x^{\a}E_1(t)u|_{0,\r}\\
=&\sup_{Y\in\Sigma(\th)}e^{\mu\Re Y}|\int^{\infty}_0\p_{YY}(\frac{1}{\sqrt{4\pi t}}e^{-\frac{(Y-Y')^2}{4t}}+
\frac{1}{\sqrt{4\pi t}}e^{-\frac{(Y+Y')^2}{4t}})\p_x^{\a}u(x,Y')dY'|_{0,\r}\\
=&\sup_{Y\in\Sigma(\th)}e^{\mu\Re Y}|\int^{\infty}_0\p_{Y'}\p_Y(-\frac{1}{\sqrt{4\pi t}}e^{-\frac{(Y-Y')^2}{4t}}+
\frac{1}{\sqrt{4\pi t}}e^{-\frac{(Y+Y')^2}{4t}})\p_x^{\a}u(x,Y')dY'|_{0,\r}\\
=&\sup_{Y\in\Sigma(\th)}e^{\mu\Re Y}|-\int^{\infty}_0\p_{Y'}(\frac{1}{\sqrt{4\pi t}}e^{-\frac{(Y-Y')^2}{4t}}+
\frac{1}{\sqrt{4\pi t}}e^{-\frac{(Y+Y')^2}{4t}})\p_{Y'}\p_x^{\a}u(x,Y')dY'|_{0,\r}\\
=&\sup_{Y\in\Sigma(\th)}e^{\mu\Re Y}|\int^{\infty}_0(\frac{1}{\sqrt{4\pi t}}e^{-\frac{(Y-Y')^2}{4t}}+
\frac{1}{\sqrt{4\pi t}}e^{-\frac{(Y+Y')^2}{4t}})\p_{Y'Y'}\p_x^{\a}u(x,Y')dY'|_{0,\r}\\
\leq &C \sup_{Y\in\Sigma(\th)}e^{\mu\Re Y}|\p_Y^2\p_x^{\a}u|_{0,\r},
\end{split}
\end{equation}
the last equality was established because the boundary term vanished due to the condition
$\p_Yu\mid_{Y=0}=0.$
Then we complete the proof of this lemma by the definition of the norm of $K^{l,\r,\th,\mu}$.
\end{proof}

Note that $\p_tE_1(t)u=\p_{YY}u$, the definition of the norm of $K^{l,\r,\th,\mu}_{\b,T}$
and the above estimate, we have the following estimates:
\begin{lemma}
Let $u\in K^{l,\r,\th,\mu}$ with $\p_Yu\mid_{Y=0}=0,$ then for any $ \b>0,$ and $T>0,$ we have that
$E_1(t)u\in K^{l,\r,\th,\mu}_{\b,T}$ and
\begin{equation}
|E_1(t)u|_{l,\r,\th,\mu,\b,T}\leq C |u|_{l,\r,\th,\mu}.
\end{equation}
\end{lemma}

\begin{lemma}\label{e11l}
Let $u\in K^{l,\r,\th,\mu},\;f\in H^{l,\r}$ satisfy $\p_Yu\mid_{Y=0}=f,$ then $E_1(t)(u-Yf)+Yf\in K^{l,\r,\th,\mu}$
and
\begin{equation}\label{e11e}
|E_1(t)(u-Yf)+Yf|_{l,\r,\th,\mu}\leq C(|u|_{l,\r,\th,\mu}+|f|_{l,\r}).
\end{equation}
This implies that for any $ \b>0,$ and $T>0,$
$E_1(t)(u-Yf)+Yf\in K^{l,\r,\th,\mu}_{\b,T}$ and
\begin{equation}
|E_1(t)(u-Yf)+Yf|_{l,\r,\th,\mu,\b,T}\leq C (|u|_{l,\r,\th,\mu}+|f|_{l,\r}).
\end{equation}
\end{lemma}

\begin{proof}
First, we check that $E_1(t)Y=Y:$
\begin{equation}
\begin{split}
E_1(t)Y=&\int^{\infty}_0(\frac{1}{\sqrt{4\pi t}}e^{-\frac{(Y-Y')^2}{4t}}+
\frac{1}{\sqrt{4\pi t}}e^{-\frac{(Y+Y')^2}{4t}})Y'dY'\\
=&\frac{1}{\sqrt{\pi}}\int_{-\frac{Y}{\sqrt{4t}}}^\infty e^{-\eta^2}(\sqrt{4t}\eta+Y)d\eta+
\frac{1}{\sqrt{\pi}}\int_{\frac{Y}{\sqrt{4t}}}^\infty e^{-\eta^2}(\sqrt{4t}\eta-Y)d\eta\\
=&\frac{1}{\sqrt{\pi}}\int_{-\infty}^{+\infty} e^{-\eta^2}(\sqrt{4t}\eta+Y)d\eta-
2\frac{1}{\sqrt{\pi}}\int_{-\infty}^{-\frac{Y}{\sqrt{4t}}}e^{-\eta^2}(\sqrt{4t}\eta+Y)d\eta\\
=&Y-\frac{\sqrt{4t}}{\sqrt{\pi}}e^{-(\frac{Y}{\sqrt{4t}})^2}-\frac{2}{\sqrt{\pi}}Y\int_{-\infty}^{-\frac{Y}{\sqrt{4t}}}e^{-\eta^2}d\eta.
\end{split}
\end{equation}
Thus,
\begin{equation}
\begin{split}
&E_1(t)(u-Yf)+Yf\\
=&E_1(t)(u-Ye^{-(\mu+1)Y}f)+E_1(t)(Ye^{-(\mu+1)Y})\\&+
\frac{2}{\sqrt{\pi}}\Big(\sqrt{t}e^{-(\frac{Y}{\sqrt{4t}})}+Y\int_{-\infty}^{-\frac{Y}{\sqrt{4t}}}e^{-\eta^2}d\eta\Big)f.
\end{split}
\end{equation}
As $\p_Y(u-Ye^{-(\mu+1)Y}f)\mid_{Y=0}=0,$ we have
\begin{equation}
|E_1(t)(u-Ye^{-(\mu+1)Y}f)|_{l,\r,\th,\mu}\leq C(|u|_{l,\r,\th,\mu}+|f|_{l,\r}),
\end{equation}
and it is easy to check that
\begin{equation}
|\frac{2}{\sqrt{\pi}}\Big(\sqrt{t}e^{-(\frac{Y}{\sqrt{4t}})}+Y\int_{-\infty}^{-\frac{Y}{\sqrt{4t}}}e^{-\eta^2}d\eta\Big)f|_{l,\r,\th,\mu}\leq C|f|_{l,\r},
\end{equation}
and
\begin{equation}
|E_1(t)(Ye^{-(\mu+1)Y}f)|_{l,\r,\th,\mu}\leq C|f|_{l,\r},
\end{equation}
then we establish \eqref{e11e}.
\end{proof}

\subsection{Estimate of $E_2$}
\begin{lemma}\label{e21l}
Let $\phi\in K^{l,\r}_{\b,T}$ with $\phi\mid_{t=0}=0,$ then $E_2\phi\in K^{l,\r,\th,\mu}_{\b,T}$ and it
satisfies
\begin{equation}
|E_2\phi|_{l,\r,\th,\mu,\b,T}\leq C|\phi|_{l,\r,\b,T}.
\end{equation}
\end{lemma}
\begin{proof}
As
\begin{equation}
\begin{split}
&\sup_{Y\in \Sigma(\th)}e^{2\mu\Re Y}|\int_0^t\frac{Y}{t-s}
\frac{e^{-\frac{Y^2}{4(t-s)}}}{\sqrt{4\pi(t-s)}}\phi(s)ds|\\
=&\sup_{Y\in \Sigma(\th)}e^{2\mu\Re Y}|\int_{\frac{Y}{\sqrt{4t}}}d\eta\phi(t-\frac{Y^2}{4\eta^2})e^{-\eta^2}|
\\\leq & C\sup_t|\phi(t)|,
\end{split}
\end{equation}
then we could also have
\begin{equation}
\begin{split}
&\sup_{Y\in \Sigma(\th)}e^{\mu\Re Y}|\int_Y^\infty dY'\int_0^t\frac{Y'}{t-s}
\frac{e^{-\frac{Y'^2}{4(t-s)}}}{\sqrt{4\pi(t-s)}}\phi(s)ds|\\
\leq & C\sup_t|\phi(t)|,
\end{split}
\end{equation}
and by $\phi(0)=0,$ it can be show that
\begin{equation}
\begin{split}
&\sup_{Y\in \Sigma(\th)}e^{\mu\Re Y}|\int_0^t\p_Y\frac{Y}{t-s}
\frac{e^{-\frac{Y^2}{4(t-s)}}}{\sqrt{4\pi(t-s)}}\phi(s)ds|\\
=&2\sup_{Y\in \Sigma(\th)}e^{\mu\Re Y}|\int_{\frac{Y}{\sqrt{4t}}}d\eta\phi(t-\frac{Y^2}{4\eta^2})e^{-\eta^2}\frac{Y}{\eta^2}|
\\\leq & C\sup_t|\phi'(t)|,
\end{split}
\end{equation}
the above estimate and the fact that $\p_t E_2\phi=\p_{YY}E_2 \phi$ show that
\begin{equation}
\begin{split}
&|E_2\phi|_{l,\r,\th,\mu,\b,T}\\
\leq &C\sum_{\a_1\leq2}\sum_{\a_2\leq l-\a_1}\sup_{0\leq t\leq T}\sup_{Y\in \Sigma(\th-\b t)}
e^{(\mu-\b t)\Re Y}|\p^{\a_1}_Y\p^{\a_2}_x E_2\phi|_{0,\r-\b t}\\
\leq &C\sum_{\a_1\leq2}\sum_{\a_2\leq l-\a_1}\sup_{0\leq t\leq T}\sup_{Y\in \Sigma(\th-\b t)}
e^{(\mu-\b t)\Re Y}\p^{\a_1}_Y E_2|\p^{\a_2}_x\phi|_{0,\r-\b t}\\
\leq &C|\phi|_{l,\r,\b,T}.
\end{split}
\end{equation}
\end{proof}

\subsection{The Estimate of $E_3$}

First, we get the estimate of $E_3(t)\p_Yu$. This term arises when we  transfer
the Robin boundary condition into the Neumann Boundary condition. For some case, we need to
transfer the derivative with respect to $Y$ to the heat kernel.

\begin{lemma}\label{e31l}
Assume that $u(t)\in K^{l,\r-\b t,\th-\b t, \mu-\b t}$, for $0\leq t\leq T.$
Then for any $ \r'<\r-\b t,\th'<\th-\b t,\mu'<\mu-\b t$, we have that
\begin{equation}
|E_3(t)\p_Y u|_{l,\r',\th',\mu'}\leq C\int^t_0ds\frac{1}{\sqrt{t-s}}|u(s)|_{l,\r',\th',\mu'}.
\end{equation}
\end{lemma}
\begin{proof}
\begin{equation}
\begin{split}
&\sup_{Y\in\Sigma(\th')}e^{\mu'\Re Y}|\p_Y^{\a_1}\p_x^{\a_2}\int^t_0ds\int^\infty_0[E(Y-Y',t-s)\\
&\qquad\qquad\qquad\qquad\qquad\qquad\quad +E(Y+Y',t-s)]\p_{Y'}u(x,Y',s)dY'|_{0,\r'}\\
=&\sup_{Y\in\Sigma(\th')}e^{\mu'\Re Y}|\int^t_0ds\int^\infty_0\p_Y^{\a_1}[E(Y-Y',t-s)\\
&\qquad\qquad\qquad\qquad\qquad\qquad+E(Y+Y',t-s)]\p_x^{\a_2}\p_{Y'}u(x,Y',s)dY'|_{0,\r'}\\
\doteq & I.
\end{split}
\end{equation}
When $\a_1=0$, we have that:

\begin{equation}
\begin{split}
I =&\sup_{Y\in\Sigma(\th')}e^{\mu'\Re Y}\mid\int^t_0ds\Big(-\int^\infty_0\p_{Y'}[E(Y-Y',t-s)
+E(Y+Y',t-s)]\\
&\qquad\qquad\qquad\qquad\qquad\qquad\cdot \p_x^{\a_2}u(x,Y',s)dY'
-2E(Y,t-s)\p_x^{\a_2}u(x,0,s)\Big)\mid_{0,\r'}\\
\leq &\sup_{Y\in\Sigma(\th')}e^{\mu'\Re Y}|\int^t_0ds\int^\infty_0\frac{1}{\sqrt{4\pi(t-s)}}\frac{-(Y-Y')}{2(t-s)}
e^{-\frac{(Y-Y')^2}{4(t-s)}}\p_x^{\a_2}u(x,Y',s)dY'|_{0,\r'}\\
&+\sup_{Y\in\Sigma(\th')}e^{\mu'\Re Y}|\int^t_0ds\int^\infty_0\frac{1}{\sqrt{4\pi(t-s)}}\frac{-(Y+Y')}{2(t-s)}
e^{-\frac{(Y+Y')^2}{4(t-s)}}\p_x^{\a_2}u(x,Y',s)dY'|_{0,\r'}\\
&+2\sup_{Y\in\Sigma(\th')}e^{\mu'\Re Y}|\int^t_0ds\frac{1}{\sqrt{\pi(t-s)}}
e^{-\frac{Y^2}{4(t-s)}}\p_x^{\a_2}u(x,0,s)|_{0,\r'}\\
\doteq & I_1+I_2+2I_3,
\end{split}
\end{equation}
where
\begin{equation}
\begin{split}
I_1=&\sup_{Y\in\Sigma(\th')}e^{\mu'\Re Y}|\int^t_0ds\int^\infty_0\frac{1}{\sqrt{4\pi(t-s)}}\frac{-(Y-Y')}{2(t-s)}
e^{-\frac{(Y-Y')^2}{4(t-s)}}\p_x^{\a_2}u(x,Y',s)dY'|_{0,\r'}\\
=&\sup_{Y\in\Sigma(\th')}e^{\mu'\Re Y}|\int^t_0ds\frac{1}{\sqrt{t-s}}
\int^\infty_{-\frac{Y}{\sqrt{4(t-s)}}}\eta e^{-\eta^2}\p_x^{\a_2}u(x,\sqrt{4(t-s)}\eta+Y,s)d\eta|_{0,\r'}\\
=&\sup_{Y\in\Sigma(\th')}|\int^t_0ds\frac{1}{\sqrt{t-s}}
\int^\infty_{-\frac{Y}{\sqrt{4(t-s)}}}\eta e^{-\eta^2-\mu'\sqrt{4(t-s)}\eta}e^{\mu'(\sqrt{4(t-s)}\eta+Y)}\\
&\;\;\;\;\;\;\;\;\;\;\;\;\;\;\;\;\;\;\;\;\;\;\;\;\;\;\;\;\;\;\;\;\;\;\;\;\;\;\;\;\;\;\;\;\;\;\;\p_x^{\a_2}u(x,\sqrt{4(t-s)}\eta+Y,s)d\eta|_{0,\r'}\\
\leq &C\int_0^tds\frac{1}{\sqrt{t-s}}\sup_{Y\in\Sigma(\th')}|e^{\mu' Y}\p_x^{\a_2}u(x,Y,s)|_{0,\r'}|\int^\infty_{-\frac{Y}{\sqrt{4(t-s)}}}\eta e^{-\eta^2-\mu'\sqrt{4(t-s)}\eta}d\eta|\\
\leq& C\int_0^tds\frac{1}{\sqrt{t-s}}|u(s)|_{l,\r',\th',\mu'},
\end{split}
\end{equation}
\begin{equation}
\begin{split}
I_2=&\sup_{Y\in\Sigma(\th')}e^{\mu'\Re Y}|\int^t_0ds\int^\infty_0\frac{1}{\sqrt{4\pi(t-s)}}\frac{-(Y+Y')}{2(t-s)}
e^{-\frac{(Y+Y')^2}{4(t-s)}}\p_x^{\a_2}u(x,Y',s)dY'|_{0,\r'}\\
=&\sup_{Y\in\Sigma(\th')}e^{\mu'\Re Y}|\int^t_0ds\frac{1}{\sqrt{t-s}}
\int^\infty_{\frac{Y}{\sqrt{4(t-s)}}}\eta e^{-\eta^2}\p_x^{\a_2}u(x,\sqrt{4(t-s)}\eta-Y,s)d\eta|_{0,\r'}\\
=&\sup_{Y\in\Sigma(\th')}|\int^t_0ds\frac{1}{\sqrt{t-s}}
\int^\infty_{\frac{Y}{\sqrt{4(t-s)}}}\eta e^{-\eta^2-\mu'\sqrt{4(t-s)}\eta+2\mu Y}e^{\mu'(\sqrt{4(t-s)}\eta-Y)}\\
&\;\;\;\;\;\;\;\;\;\;\;\;\;\;\;\;\;\;\;\;\;\;\;\;\;\;\;\;\;\;\;\;\;\;\;\;\;\;\;\;\;\;\;\;\;\;\;\p_x^{\a_2}u(x,\sqrt{4(t-s)}\eta-Y,s)d\eta|_{0,\r'}\\
\leq &C\int_0^tds\frac{1}{\sqrt{t-s}}\sup_{Y\in\Sigma(\th')}|e^{\mu' Y}\p_x^{\a_2}u(x,Y,s)|_{0,\r'}|\int^\infty_{\frac{Y}{\sqrt{4(t-s)}}}\eta e^{-\eta^2-\mu'\sqrt{4(t-s)}\eta+2\mu' Y}d\eta|\\
\leq& C\int_0^tds\frac{1}{\sqrt{t-s}}|u(s)|_{l,\r',\th',\mu'},
\end{split}
\end{equation}
\begin{equation}
\begin{split}
I_3=&\sup_{Y\in\Sigma(\th')}e^{\mu'\Re Y}|\int^t_0ds\frac{1}{\sqrt{\pi(t-s)}}
e^{-\frac{Y^2}{4(t-s)}}\p_x^{\a_2}u(x,0,s)|_{0,\r'}\\
\leq & C\int_0^tds\frac{1}{\sqrt{t-s}}|\p^{\a_2}_xu(x,0,s)|_{0,\r'}\sup_{Y\in\Sigma(\th')}|e^{\mu' Y}e^{-\frac{Y^2}{4(t-s)}}|\\
\leq & C\int_0^tds\frac{1}{\sqrt{t-s}}|\p^{\a_2}_xu(x,0,s)|_{0,\r'}\\
\leq& C\int_0^tds\frac{1}{\sqrt{t-s}}|u(s)|_{l,\r',\th',\mu'}.
\end{split}
\end{equation}
When $\a_1=1,$ we do not need to transfer the derivative with respect to $Y$, then we could have that
\begin{equation}
\begin{split}
I
=&\sup_{Y\in\Sigma(\th')}e^{\mu'\Re Y}\mid\int^t_0ds\int^\infty_0\p_{Y}[E(Y-Y',t-s)\\
&\;\;\;\;\;\;\;\;\;\;\;\;\;\;\;\;\;\;\;\;\;\;\;\;\;\;\;\;\;\;\;\;\;\;\;\;\;\;\;\;\;\;
+E(Y+Y',t-s)]\p_{Y'}\p_x^{\a_2}u(x,Y',s)dY'|_{0,\r'}\\
\leq & C\int_0^tds \frac{1}{\sqrt{t-s}}|u(s)|_{l,\r',\th',\mu'}.
\end{split}
\end{equation}
when $\a_1=2,$
\begin{equation}
\begin{split}
I=&\sup_{Y\in\Sigma(\th')}e^{\mu'\Re Y}\mid\int^t_0ds\int^\infty_0\p_{Y}^2[E(Y-Y',t-s)\\
&\;\;\;\;\;\;\;\;\;\;\;\;\;\;\;\;\;\;\;\;\;\;\;\;\;\;\;\;\;\;\;\;\;\;\;\;\;\;\;\;\;\;
+E(Y+Y',t-s)]\p_{Y'}\p_x^{\a_2}u(x,Y',s)dY'|_{0,\r'}\\
=&\sup_{Y\in\Sigma(\th')}e^{\mu'\Re Y}\mid\int^t_0ds\int^\infty_0\p_{Y'}\p_{Y}[-E(Y-Y',t-s)\\
&\;\;\;\;\;\;\;\;\;\;\;\;\;\;\;\;\;\;\;\;\;\;\;\;\;\;\;\;\;\;\;\;\;\;\;\;\;\;\;\;\;\;
+E(Y+Y',t-s)]\p_{Y'}\p_x^{\a_2}u(x,Y',s)dY'|_{0,\r'}\\
=&\sup_{Y\in\Sigma(\th')}e^{\mu'\Re Y}\mid\int^t_0ds\Big(-\int^\infty_0\p_{Y}[-E(Y-Y',t-s)\\
&\;\;\;\;\;\;\;\;\;\;\;\;\;\;\;\;\;\;\;\;\;\;\;\;\;\;\;\;\;\;\;\;\;\;\;\;\;\;\;\;\;\;
+E(Y+Y',t-s)]\p_{Y'}^2\p_x^{\a_2}u(x,Y',s)dY'\Big)|_{0,\r'}\\
\leq& C\int_0^tds\frac{1}{\sqrt{t-s}}|u(s)|_{l,\r',\th',\mu'},
\end{split}
\end{equation}
here the boundary term vanishes because of the two factors of $E$. Then we complete
the proof of this lemma.
\end{proof}

Next, we estimate $E_3(t)u$.
\begin{lemma}\label{e32l}
Assume that $u(t)\in K^{l,\r-\b t,\th-\b t, \mu-\b t}$, for $0\leq t\leq T,$ satisfying $\p_Yu\mid_{Y=0}=0.$ Then
$\forall \r'<\r-\b t,\th'<\th-\b t,\mu'<\mu-\b t$, we have that
\begin{equation}
|E_3(t)u|_{l,\r',\th',\mu'}\leq \int^t_0 ds|u(s)|_{l,\r',\th',\mu'}\leq C|u|_{l,\r,\th,\mu,\b,T}.
\end{equation}
\end{lemma}
\begin{proof}
\begin{equation}
\begin{split}
&\sup_{Y\in\Sigma(\th')}e^{\mu'\Re Y}|\p_Y^{\a_1}\p_x^{\a_2}\int^t_0ds\int^\infty_0[E(Y-Y',t-s)\\
&\qquad\qquad\qquad\qquad\qquad\qquad\;\;\;\;+E(Y+Y',t-s)]u(x,Y',s)dY'|_{0,\r'}\\
=&\sup_{Y\in\Sigma(\th')}e^{\mu'\Re Y}|\int^t_0ds\int^\infty_0[E(Y-Y',t-s)\\
&\qquad\qquad\qquad\qquad\qquad\quad+E(Y+Y',t-s)]\p_Y^{\a_1}\p_x^{\a_2}u(x,Y',s)dY'|_{0,\r'}\\
\end{split}
\end{equation}
for $\a_1=1,$ the boundary term vanishes because of the two factors of $E$, for $\a_1=2,$
the boundary term vanishes because of the condition $\p_Yu\mid_{Y=0}=0.$ Then
\begin{equation}
\begin{split}
&\sup_{Y\in\Sigma(\th')}e^{\mu'\Re Y}|\int^t_0ds\int^\infty_0E(Y-Y',t-s)
\p_Y^{\a_1}\p_x^{\a_2}u(x,Y',s)dY'|_{0,\r'}\\
=&\sup_{Y\in\Sigma(\th')}e^{\mu'\Re Y}|\int^t_0ds\int^\infty_0\frac{1}{\sqrt{4\pi(t-s)}}
e^{-\frac{(Y-Y')^2}{4(t-s)}}
\p_Y^{\a_1}\p_x^{\a_2}u(x,Y',s)dY'|_{0,\r'}\\
=&\sup_{Y\in\Sigma(\th')}e^{\mu'\Re Y}|\int^t_0ds
\int^\infty_{-\frac{Y}{\sqrt{4(t-s)}}}\eta e^{-\eta^2}\p_x^{\a_2}u(x,\sqrt{4(t-s)}\eta+Y,s)d\eta|_{0,\r'}\\
\leq &\sup_{Y\in\Sigma(\th')}|\int^t_0ds
\int^\infty_{-\frac{Y}{\sqrt{4(t-s)}}}\eta e^{-\eta^2-\mu'\sqrt{4(t-s)}\eta}e^{\mu'(\sqrt{4(t-s)}\eta+Y)}\p_x^{\a_2}u(x,\sqrt{4(t-s)}\eta+Y,s)d\eta|_{0,\r'}\\
\leq &C \int^t_0 ds|u(s)|_{l,\r',\th',\mu'}.
\end{split}
\end{equation}
Similarly, we could get the estimate of the other factor of $E_3,$ and complete the proof of this lemma.
\end{proof}

By the definition of the norm of $K^{l,\r,\th,\mu}_{\b,T}$ and the fact that
$\p_t E_3(t)u=\p_{YY}E_3(t)u+u,$ we could have the following estimate:
\begin{lemma}
Let $u\in K^{l,\r,\th,\mu}_{\b,T},$ then $E_3(t) u\in K^{l,\r,\th,\mu}_{\b,T},$ and
\begin{equation}
|E_3(t)u|_{l,\r,\th,\mu,\b,T}\leq C|u|_{l,\r,\th,\mu,\b,T}
\end{equation}
\end{lemma}

\section{Proof of proposition \ref{pro1}}

The solution $u$ of (\ref{p2e})
can be written as:
\begin{equation}
\begin{split}
u=&E_1(t)u_0+E_2(U(x,t))+E_3(t)(K(u,t)+2\p_Y u)\\
=&E_1(t)(u_0-YU(x,0))+YU(x,0)+E_2(U(x,t)-U(x,0))\\&+E_3(t)(K(u,t)+2\p_Y u)\\
\doteq& \mathcal{U}+E_3(t)(K(u,t)+2\p_Y u)\\
\doteq & F(u,t).
\end{split}
\end{equation}

The rest is devoted to proving that $F(u,t)$ satisfies all the conditions
of the ACK Theorem with $X=K^{l,\r,\th,\mu+1}$.

It is obvious that $F$ satisfies the first condition of the ACK theorem
which says that $F$ has some continuation with respect to time $t$. The
following proposition shows that $F$ satisfies the second condition (\ref{ack1})
of the ACK theorem.
\begin{proposition}
Assume that $u_0(x,Y)\in K^{l,\r,\th,\mu+1},U\in K^{l,\r}_{\b,T},$ satisfy
$\p_Y u_0\mid_{Y=0}=U(x,0),$
then we have that
$\mathcal{U}\in K^{l,\r,\th,\mu+1}_{\b,T}$ satisfying
\begin{equation}
|\mathcal{U}|_{l,\r,\th,\mu+1,\b,T}\leq C(|u_0|_{l,\r,\th,\mu+1}+|U|_{l,\r,\b,T}).
\end{equation}
\end{proposition}
As
\begin{equation}
\mathcal{U}=E_1(t)(u_0-YU(x,0))+YU(x,0)+E_2(U(x,t)-U(x,0)),
\end{equation}
then by Lemma \ref{e11l} and Lemma \ref{e21l}, we could obtain this proposition.

Next we shall prove that the operator $F$ satisfies the
last condition (\ref{ack2}) of the ACK Theorem. That is the following proposition.
\begin{proposition}\label{pro1}
$0<\r'<\r(s)\leq \r_0-\b_0 s, 0<\th'<\th(s)\leq \th_0-\b_0 s,
0<\mu'+1<\mu(s)+1\leq \mu_0+1-\b_0 s,$ and
$\forall u_1(t),u_2(t)\in K^{l,\r-\b_0 t,\th-\b_0 t, \mu+1-\b_0 t}$, satisfy
\begin{equation}
\sup_{0\leq t\leq T}|u_j|_{l,\r_0-\b_0 t,\th_0-\b_0 t,\mu_0+1-\b_0 t}\leq R,   \;\;\;\;\;j=1,2,
\end{equation}
we have
\begin{equation}
\begin{split}
&|F(t,u_1)-F(t,u_2)|_{l,\r',\th',\mu'+1}\\
\leq&C\int^t_0 ds(\frac{|u_1-u_2|_{\r(s),\th',\mu'+1}}{\r(s)-\r'}+\frac{|u_1-u_2|_{\r',\th(s),\mu'+1}}{\th(s)-\th'}\\
&\;\;\;\;\;\;\;\;\;\;+\frac{|u_1-u_2|_{\r',\th',\mu(s)+1}}{\mu(s)-\mu'}+\frac{|u_1-u_2|_{\r',\th',\mu'+1}}{\sqrt{t-s}}).
\end{split}
\end{equation}
\end{proposition}
{\it Proof:} $\;$ As
\begin{equation}
F(u,t)=\mathcal{U}+E_3(t)(K(u,t))+2E_3(t)\p_Y u,
\end{equation}
one can check that $\p_Y F(u,t)\mid_{Y=0}=U(x,t).$ We can get the
solution of $u=F(u,t)$ by iteration on the set of functions satisfying
$\p_Y u\mid_{Y=0}=U(x,t).$  As a result, we could assume that
\begin{equation}
\big(K(u_1,t)-K(u_2,t)\big)\mid_{Y=0}=0.
\end{equation}
 Then Proposition \ref{pro1} is a direct conclusion of Lemma \ref{e31l}, Lemma \ref{e32l} and the following Cauchy estimate of
the operator $K$.
\begin{proposition}\label{pro2}
Suppose $u_1(t),u_2(t)\in K^{l,\r_0-\b_0 t,\th_0-\b_0 t, \mu_0-\b_0 t}$, for $0\leq t\leq T.$ Then
$\forall \r'<\r''<\r_0-\b_0 t,\th'<\th''<\th_0-\b_0 t,\mu'<\mu''<\mu_0-\b_0 t$, we have that
\begin{equation}
\begin{split}
&|K(u_1,t)-K(u_2,t)|_{l,\r',\th',\mu'}\\
\leq &C\Big(\frac{|u_1-u_2|_{\r'',\th',\mu'+1}}{\r''-\r'}+\frac{|u_1-u_2|_{\r',\th'',\mu'+1}}{\th''-\th'}\\
&\;\;\;\;\;\;\;\;\;\;+\frac{|u_1-u_2|_{\r',\th',\mu''+1}}{\mu''-\mu'}+|u_1-u_2|_{\r',\th',\mu'+1}\Big).
\end{split}
\end{equation}
\end{proposition}
This proposition could be proved by using Proposition \ref{proposition1}, Lemma \ref{lemma1} and
Lemma \ref{lemma2}. The proof  is the same as the proof of Proposition 5.7 in \cite{sc}, only noticing
the fact that  $e^Yu\in K^{l,\r,\th,\mu}$ if $u\in K^{l,\r,\th,\mu+1}.$ We omit the details here.

Then by the ACK Theorem, there exists a unique solution $u$ of the equation satisfies
\begin{equation}
|u(t)|_{l,\r_1-\b_1 t,\th_1-\b_1 t,\mu_1-\b_1 t}\leq C,\quad
\end{equation}
By the equation (\ref{p2e}),we have
\begin{equation}
\p_tu=\p_{YY}u+K(u,t)+2\p_Yu,
\end{equation}
then we obtain $u\in K^{l,\r_1,\th_1+1,\mu_1}_{\b_1,T}$ and complete the proof of Proposition
\ref{pro21}.\\

\noindent{\bf Acknowledgment:}
Ning Jiang was supported by a grant from the National Natural Science Foundation of China under contract  No. 11171173. Yutao Ding was supported by the postdoctoral funding of Mathematical Science Center of Tsinghua University.


\begin{thebibliography}{9999}
\small

\bibitem{al}R. Alexandre, Y. Wang, C. Xu and T. Yang {\it Well-posedness of the Prandtl Equation in Sobolev Spaces.}
 arXiv:1203.5991vl.

\bibitem{1}  H. Beir$\tilde{a}$o Veiga and  F. Crispo,  {\it Sharp Inviscid Limit Results under Navier Type Boundary Conditions. An $L^p$ theory}.
 J. Math. Fluid Mech. {\bf 12} (2010), no. 3, 397-411.

\bibitem{2}  H. Beir$\tilde{a}$o Veiga, F. Crispo and C. R. Grisanti, {\it Reducing slip boundary value problems from the half to the whole space. Applications to inviscid limits and to non-Newtonian fluids.}  J. Math. Anal. Appl. {\bf 377} (2011), no. 1, 216-227.

\bibitem{ca} L. Caffarelli, W. E, {\it Separation of steady boundary layers.} unpublished, 1995

\bibitem{clopau}  T. Clopeau, A. Mikeli¡äc, R. Robert, {\it On the vanishing viscosity limit for the 2D incompressible Navier-Stokes equations
with the friction type boundary condition.}   Nonlinearity {\bf 11} (1998), no. 6, 1625-1636.

\bibitem{c} L. Crocco, {\it Sullo strato limite laminare nei gas lungo una lamina plana.}  Rend. Math. Appl. Ser. {\bf 5}. (1941), 138-152.

\bibitem{e} W. E,  {\it Boundary layer theory and the zero-viscosity limit of the Navier-Stokes equation.}
 Acta Math. Sin. (Engl.\,Ser.) {\bf 16},  2 (2000), 207-218.

\bibitem{gd}  D. G\'{e}rard-Varet and E. Dormy, {\it On the ill-posedness of the Prandtl equation.}  J. Amer. Math. Soc. {\bf 23} (2010)
, no. 2, 591-609.

\bibitem{gd2} D. G\'{e}rard-Varet and T. Nguyen. {\it Remarks on the ill-posedness of the Prandtl equation}  Asymptot. Anal. {\bf 77} (2012), no. 1-2, 71-88.

\bibitem{gd3} D. G\'{e}rard-Varet, N. Masmoudi, {\it Well-posedness for the Prandtl system
without analyticity or monotonicity.}  arXiv:1305.0221vl.


\bibitem{gt} Y. Guo and T. Nguyen. {\it A note on the prandtl layers.}  Comm. Pure Appl. Math. {\bf 64} (2011), no. 10, 1416-1438

\bibitem{plana}  D. Iftimie and G. Planas, {\it Inviscid limits for the Navier-Stokes equations with Navier friction boundary conditions.}
 Nonlinearity  {\bf 19} (2006), no. 4, 899-918.

\bibitem{plana1}D. Iftimie and F. Sueur, {\it Viscous boundary layers for the Navier-Stokes equations with the Navier slip conditions.} Arch.
Rational Mech. Anal.  {\bf 199} (2011), no. 1, 145-175.

\bibitem{KV} I. Kukavica and V. Vicol, {\it On the local existence of analytic solutions to the Prandtl boundary layer equations.}  Commun. Math. Sci. {\bf 11} (2013), no. 1, 269-292.

\bibitem{s}M. Lombardo, M. Cannone and M. Sammartino, {\it Well-posedness of the boundary layer equations.}  SIAM J. Math. Anal. {\bf 35} (2003), no. 4, 987-1004

\bibitem{lopes}  M.C. Lopes Filho,  H.J. Nussenzveig Lopes and  G. Planas,  {\it On the inviscid limit for two-dimensional incompressible flow
with Navier friction condition.}  SIAM J. Math. Anal. {\bf 36} (2005), no. 4, 1130-1141

\bibitem{ma}N. Masmoudi, {\it Remarks about the invisicid limit of Navier-Stokes system.}  Comm. Math. Phys. {\bf 270} (2007), no. 3, 777-788.

\bibitem{rousset}  N. Masmoudi and F. Rousset,
{\em Uniform Regularity for the Navier-Stokes Equations with Navier Boundary Condition.} Arch. Ration.
Mech. Anal. {\bf 203} (2012), no. 3, 529-575

\bibitem{wong}N. Masmoudi, T. Wong, {\it Local-in-time existence and uniquness od solutions to the Prandtl
equations by energy method. }  To appear in Comm. Pure Appl. Math.

\bibitem{mae}Y. Maekawa, {\it On the inviscid limit problem of the vorticity equations for viscous incompressible
flows in the half plane.} Preprint 2013.

\bibitem{navier} Navier,  {\it Sur les lois de l¡¯equilibre et du mouvement des corps ¡äelastiques.}  Mem. Acad. R. Sci. Inst.
France. {\bf 6}, 369 (1827)

\bibitem{os} O.A. Oleinik and V.N. Samokhin,  {\it Mathematical Models in Boundary
 Layer  Theory.} Applied Mathematics and Mathematical Computation, {\bf 15}. Chapman \& Hall/CRC, Boca Raton, FL, 1999.

\bibitem{p} L. Prandtl, {\it Uber fl\"{u}ssigkeits-bewegung bei sehr kleiner reibung.} Actes du 3me Congr$\acute{e}$s
international dse Math$\acute{e}$maticiens, Heidelberg. Teubner,leipzig, 1904, pp. 484-491.

\bibitem{sc} M. Sammartino, and R. E. Caflisch, {\it Zero viscosity limit for analytic solutions, of the Navier-Stokes equation on a half-space. I. Existence for Euler and Prandtl equations.}  Comm. Math. Phys. {\bf 192} (1998), no. 2, 433-461.

\bibitem{sc1} M. Sammartino,and R. E. Caflisch, {\it Zero viscosity limit for analytic solutions of the Navier-Stokes equation on a half-space. II. Construction of the Navier-Stokes solution. }  Comm. Math. Phys. {\bf 192} (1998), no. 2, 463-491.

\bibitem{wangxin} X. Wang, Y. Wang and Z. Xin, {\it Boundary Layer in incompressible Navier-Stokes equations with Navier boundary conditions
for vanishing viscosity limit.}  Commun. Math. Sci. {\bf 8} (2010), no. 4, 965-998.

\bibitem{xin} Y. Xiao and  Z. Xin,  {\it On the vanishing viscosity limit for the Navier-Stokes equations with a slip boundary condition.}
 Comm. Pure Appl. Math. {\bf 60} (2007), no. 7, 1027-1055.

\bibitem{xin1} Y. Xiao and   Z. Xin,  {\it Remarks on the vanishing viscosity limit for the Navier-Stokes equations with a slip boundary
condition.}  Chin. Ann. Math. Ser. B {\bf 32} (2011), no. 3, 321-332.

\bibitem{xz} Z. Xin, and L. Zhang,  {\it On the global existence of solutions to the Prandtl's system.}   Adv. Math. {\bf 181} (2004), no. 1, 88-133.



\end{thebibliography}
\end{document}